\documentclass[12pt,a4paper]{amsart}

 \usepackage{amssymb, amstext, amscd, amsmath, amsfonts, amsthm, amscd, color}

\usepackage{tikz,mathdots,enumerate}

\usepackage{graphicx, array, blindtext}
\usepackage{url}
\usepackage{tikz}
\usetikzlibrary{shapes.geometric, calc}
\usepackage{caption}
\usepackage{array}
\usepackage{epsfig}
\usepackage{eucal}
\usepackage{latexsym}
\usepackage{mathrsfs}
\usepackage{textcomp}
\usepackage{verbatim}

\newtheorem{thm}{Theorem}[section]

\newtheorem{lemma}[thm]{Lemma}

\newtheorem{prop}[thm]{Proposition}

\numberwithin{equation}{section}

\theoremstyle{definition}

\newtheorem{rem}[thm]{Remark}
\newtheorem{example}[thm]{Example}
\newtheorem{definition}[thm]{Definition}
\newcommand{\bC}{{\mathbb{C}}}

\newcommand{\bN}{{\mathbb{N}}}
\newcommand{\bQ}{{\mathbb{Q}}}
\newcommand{\bR}{{\mathbb{R}}}
\newcommand{\bT}{{\mathbb{T}}}

\newcommand{\bH}{{\mathbb{H}}}
  \newcommand{\A}{{\mathcal{A}}}
  
  \newcommand{\C}{{\mathcal{C}}}
  \newcommand{\D}{{\mathcal{D}}}
  
  \newcommand{\F}{{\mathcal{F}}}
  \newcommand{\G}{{\mathcal{G}}}

\renewcommand{\L}{{\mathcal{L}}}
  \newcommand{\M}{{\mathcal{M}}}
  \newcommand{\N}{{\mathcal{N}}}

\renewcommand{\S}{{\mathcal{S}}}
  
  \newcommand{\U}{{\mathcal{U}}}

\usepackage{verbatim}
\usepackage{tikz}
\usetikzlibrary{calc}

\begin{document}

%\hfill {\color{blue}{15th December 2013}}

\title[The parabolic algebra on Banach spaces]{The parabolic algebra on $L^p$ spaces}

% \title[The operator algebra  generated by the three semigroups]{The %operator algebra on $L^2(\bR)$ generated by the  translation, multiplication and dilation semigroups}

% \title[The operator algebra  generated by the three semigroups]{The %operator algebra on $L^2(\bR)$ generated by the  translation, %multiplication and dilation semigroups}

\author[E. Kastis]{E. Kastis}
%\thanks{Supported by EPSRC grant  EP/J008648/1.}
%\emph{Crystal frameworks, almost periodic rigidity}
\address{Dept.\ Math.\ Stats.\\ Lancaster University\\
Lancaster LA1 4YF \\U.K. }

\email{l.kastis@lancaster.ac.uk}
%\email{s.power@lancaster.ac.uk}
%\begin{center}

%\thanks{thanks ?}


\begin{abstract}
The parabolic algebra was introduced by Katavolos and Power, in 1997, as the SOT - closed operator algebra acting on $L^2(\bR)$ that is generated by the translation and multiplication semigroups. In particular, they proved that this algebra is reflexive and is equal to the Fourier binest algebra, that is, to the algebra of operators that leave invariant the subspaces in the Volterra nest and its analytic counterpart.

We prove that a similar result holds for the corresponding algebras acting on $L^p(\bR)$, where $1< p<\infty$. In the last section, it is also shown that the reflexive closures of the Fourier binests on $L^p(\bR)$ are all order isomorphic for $1<p<\infty$.
\end{abstract}
\date{}

\maketitle
%\tableofcontents

\section{Introduction}
Let $\{D_\mu, \mu \in \bR\}$ and $\{M_\lambda, \lambda \in \bR\}$  be the groups of translation and multiplication respectively acting on the Hilbert space $L^2(\bR)$, given by 
\[
D_\mu f(x) = f(x-\mu),\quad  M_\lambda f(x) = e^{i\lambda x}f(x).
\]
 It is well-known that these 1-parameter unitary groups are continuous in the strong operator topology (SOT), that they provide an irreducible representation of the Weyl-commutation relations, $M_\lambda D_\mu
= e^{i\lambda \mu} D_\mu M_\lambda$, and that the SOT-closed operator algebra they generate is the von Neumann algebra  $B(L^2(\bR))$ of all bounded operators. (See Taylor \cite{tay}, for example.) On the other hand it was shown by Katavolos and Power in \cite{kat-pow-1} that the strongly closed nonselfadjoint operator algebra  generated by the semigroups for $\mu \geq 0$ and $ \lambda \geq 0$ is a reflexive algebra, in the sense of Halmos \cite{rad-ros}, containing no self-adjoint operators, other than real multiples of the identity, and containing no nonzero finite rank operators. We consider here the operator algebras  $\A_{par}^p$ on $L^p(\bR)$ for $1<p<\infty$, which are similarly generated by the same semigroups, viewed now as bounded operators on $L^p(\bR)$. Our main result is that $\A_{par}^p$ is also reflexive and, moreover, is equal to $\A_{FB}^p$, the algebra of operators that leave invariant each subspace in the Fourier binest $\L_{FB}^p$  of closed subspaces given by
\[
 \L_{ FB}^p=  \{0\}\cup\{L^p[t,\infty), t\in\bR\}\cup \{e^{i\lambda x}H^p(\bR), \lambda\in\bR\}\cup\{L^p(\bR)\}
\]
where  $H^p(\bR)$ is the usual Hardy space for the upper half plane. This lattice of closed subspaces is a binest equal to the union of two complete continuous nests of closed subspaces.
\bigskip

Although the reflexivity of non selfadjoint operator algebras has been studied intensively over the last fifty years, the developments have been largely confined within the limits of operator algebras acting on Hilbert spaces. For example, general nest algebras, being the most characteristic class of reflexive noncommutative non selfadjoint operator algebras since they were introduced by Ringrose in 1965 \cite{rin},  have a well-developed general theory on Hilbert spaces (Davidson \cite{dav}). However, only sporadic results can be found for nest algebras on Banach spaces (see \cite{west}, \cite{spiv}, \cite{bou-dav}). 

On the other hand, the study of the reflexivity of non-selfadjoint algebras that are generated by semigroups of operators was begun by Sarason in 1966 \cite{sar}, where he proved that $H^\infty(\bR)$, viewed as a multiplication algebra on $H^2(\bR)$, is reflexive. Since then, several results about 2-parameter Lie semigroup algebras have been obtained. One of the aims in the analysis of reflexivity and related properties is to understand better the algebraic structure of these somewhat mysterious algebras. Establishing reflexivity can provide a route to constructing operators in the algebra and thereby deriving further algebraic properties. As we stated above, the reflexivity of $A_{par}^2$, known as the parabolic algebra, was obtained in \cite{kat-pow-1}. Furthermore, Levene and Power have shown (\cite{lev-pow-1}) the reflexivity of an analogous hyperbolic algebra\cite{kat-pow-2}, the algebra generated by the multiplication and dilation semigroups on $L^2(\bR)$. The latter semigroup is given by the operators $V_t$, with
\[V_tf(x)=e^{t/2}f(e^t x),\]
for $t\geq0$. Recently, Power and the author proved that the triple semigroup operator algebra that is generated by the translation, dilation and multiplication semigroups is also reflexive \cite{kas-pow}.

A complication in establishing the reflexivity of the parabolic and hyperbolic algebras on Hilbert space is the absence of an approximate identity of finite rank operators,
a key device in the theory of nest algebras (Erdos and Power \cite{erd-pow}, Davidson \cite{dav}). However, it was shown that the subspace of Hilbert-Schmidt operators is dense for both algebras and that these operators could be used as an alternative. On the other hand, we note that different reflexivity techniques were given in \cite{ano-kat-tod}, where the authors make use of direct integral decomposition arguments.
In our case, we define a right ideal of what we refer to as $(p,q)$-integral operators which we show is able to play the role of the (two-sided) ideal of Hilbert - Schmidt operators. As a substitute for the techniques of Hilbert space geometry and tensor product identifications used in \cite{kas-pow}, \cite{kat-pow-1}, \cite{lev-pow-1}, we make use of more involved measure theoretic arguments appropriate for the $(p,q)$-integrable operators.

We also obtain a number of properties of the parabolic algebra on $L^p(\bR)$, that correspond to the classical case.
Namely, $\A_{par}^p$ is antisymmetric (or triangular \cite{kad-sin}), in an appropriate sense, and $\A_{par}^p$ contains no non-trivial finite rank operators. Futhermore, the lattice of $\A_{par}^p$ is order isomorphic to the lattice of $\A_{par}^2$ for all $1<p<\infty$. 
\medskip
\section{Preliminaries}
\subsection {The Hardy space $H^p(\bR),\,p\in[1,\infty)$}
We start with two elementary density lemmas for the Hardy spaces $H^p(\bR)$ on the line, for $p\in(1,\infty)$.  The details of the theory of Hardy spaces can be found in \cite{hof}. 

For each $u$ in the open upper half plane $\bH^+$ of $\bC$ let
\[b_u(x)=\frac{1}{x+u},\,x\in\bR.\]
Then $b_u$ lies in $H^p(\bR)$, for every $p\in(1,\infty)$.
\begin{lemma}
The linear spans of the sets $\D_1= \{b_u|u\in\bH^+\},\,\D_2=\{b_ub_w|u, w\in\bH^+\}$ are both dense in $H^p(\bR)$, for $1<p<\infty$.
\end{lemma}
\begin{proof}
Fix some $p\in(1,\infty)$ and suppose that there exists some $f\in H^p(\bR)$ that does not lie in the closed linear span of $\D_1$. Then by the Hahn - Banach theorem, there is some function $g\in L^q(\bR)$, such that 
$\int_{\bR} b_u g=0,$ for all $u\in \bH^+$, and $\int_{\bR} fg\neq 0$. But 
\[\int_{\bR} b_u g=0, \,\forall u\in \bH^+ \Leftrightarrow \int_{\bR} \frac{g(x)}{x+u}dx=0, \,\forall u\in \bH^+
\Leftrightarrow g\in H^q(\bR) \]
Hence $fg\in H^1(\bR)$, so $\int_{\bR} fg=0$, which gives a contradiction. 

Now, for any distinct $u,w\in\bH^+$, observe that
\[ b_u(x)b_w(x)=\frac{b_u(x)-b_w(x)}{w-u}. \]
Define 
\[h_n=(ni-u)b_u b_{ni}=b_u-b_{ni}. \]
Since $h_n\rightarrow b_u$ pointwise, as $n\rightarrow \infty$, and $|h_n(x)|\leq |h_1(x)|$, for all $x\in\bR$,
it follows from dominated convergence that $h_n\stackrel{\|\cdot\|_p}{\rightarrow} b_u$.
Therefore, given $u\in\bH^+$, the function $b_u$ lies in the closed linear span of $\D_2$, so by the first part of the lemma, the proof is complete. 
\end{proof} 
\begin{lemma}\label{dense}
Let $\bH^+_\bQ=\{u\in \bH^+ : u=x+iy,\text{ where } x,y\in\bQ\}$. For every $t\in \bR$ the countable set
\[\Lambda_t=\{b_uD_tb_w|u,w\in \bH^+_\bQ\}\]
is dense in $H^p(\bR)$, for every $p\in(1,\infty)$.
\end{lemma}
\begin{proof}
Observe first that
\[ D_tb_w(x)=\frac{1}{(x-t)+w}= \frac{1}{x+(w-t)}=b_{w-t}(x),\,x\in\bR.\]
Since $\bQ$ is dense in $\bR$, the rest of the proof is a simple application of dominated convergence.
\end{proof}
\medskip
The Beurling theorem in the $L^p$-setting will be useful for the determination of the lattice of $A_{par}^p$ (\cite{hel}, \cite{redd}). Applying the isometric isomorphism
\[\Phi_p: L^p(\bT)\rightarrow L^p(\bR): (\Phi_p f)(t)=\left(\frac{1}{\sqrt{\pi}(1-it)}\right)^{2/p}f\left(\frac{1+it}{1-it}\right),\,\,t\in\bR, \]
 we can state the theorem on the real line, instead of the unit circle. Note that $\Phi_p$ restricts to an isomorphism of $H_p$ -spaces \cite{koo}. 
\begin{thm}
 Given $p\in (1,\infty)$, let $\M$ be a closed subspace of $L^p(\bR)$ which is invariant under the semigroup $\{M_\lambda,\lambda\geq 0\}$. Then $\M$ is either of the form $L^p(E)$ for some Borel subset $E\subseteq \bR$ or $\M$ is equal to $\phi H^p(\bR)$ for some unimodular function $\phi$.
\end{thm}  
\medskip

\subsection{The space $L^p( \bR ; L^q(\bR))$, for $1< p,q<\infty $}
We now introduce some notation and terminology associated with the classical space $L^p( \bR ; L^q(\bR))$. This space is a space of kernel functions for what we refer to as the $(p,q)$-integrable operators. For more details, we refer the reader to \cite{pet},\cite{pis}.

Let $p,q\in[1,+\infty]$. Define $\S( \bR ; L^q(\bR))$ be the space of measurable simple functions; i.e. the functions $f: \bR\rightarrow L^q(\bR)$ taking only finitely many values :
\[ f(x)=\sum\limits_{k=1}^{n}\chi_{A_k}(x)g_k,
\]
where $\{A_k\}_{k=1,\dots,n}$ is a finite family of Borel measurable pairwise disjoint sets and where\\ 
$g_k\in L^q(\bR)$. 
\begin{definition}
A function $f: \bR\rightarrow L^q(\bR)$ is said to be strongly measurable if there is a sequence $(f_n)$ in $\S( \bR ; L^q(\bR))$, tending to $f$ pointwise. Also, $f$ is weakly measurable, if given $\omega\in (L^q(\bR))^\ast$ the function $t\mapsto\omega(f(t))$ is Borel measurable.   
\end{definition}

The relationship between strong and weak measurability is given by the following theorem of Pettis \cite{pet}, who introduced the notion of almost separably valued functions. 
\begin{definition}
Let $1\leq q\leq \infty$. A function $f: \bR \rightarrow  L^q(\bR)$ is  almost separably valued, if there is a conull Borel set $A\subseteq \bR$, such that $f(A)$ is separable.  
\end{definition}
\begin{thm}
A function $f: \bR \rightarrow L^q(\bR)$ is strongly measurable if and only if it is weakly measurable and almost separably valued.
\end{thm}
\begin{example} \label{ex1}
Define $f:\bR\rightarrow L^\infty(\bR)$ by $f(x)=\chi_{(-\infty,x]}$. Then $f$ is not almost separably valued, and hence not strongly measurable, since $\|f(x)-f(t)\|_\infty=1$ for $x\neq t$. However, for $q\in(1,\infty)$, the function $g:\bR\rightarrow L^q(\bR)$, given by $g(x)=\chi_{(-\infty,x]} f$, where $f\in L^q(\bR)$, is strongly measurable. To see this, note that $L^q(\bR)$ is separable and given $\omega\in L^p(\bR)$, where $p$ is the conjugate exponent of $q$, we have
\[
\omega(g(x))=\int_\bR \omega(y) \chi_{(-\infty,x]}(y) f(y) dy=\int_{-\infty}^x \omega(y) f(y) dy,
\]
which is measurable, being the limit of absolutely continuous functions.
\end{example}
 The definition of $L^p$ spaces of $L^q$-valued functions is analogous to the case of scalar valued functions. Define
  $L^p( \bR ; L^q(\bR))$ as the set of equivalence classes (modulo equality for almost every $x\in\bR$) of strongly measurable functions $f$ that satisfy 
 $\left(\int_{\bR} \|f(x)\|_q^p dx\right)^{1/p}<\infty$ for $1\leq p<\infty$, and $\text{esssup}\|f(\cdot)\|_q$ for $p=\infty$. Each of the above spaces endowed with the respective norm

\[
 \|f\|_{p,q}=\left(\int_{\bR} \|f(x)\|_q^p dx\right)^{1/p}\,, \text{ for } p\in[1,\infty), \] 
\[ 
 \|f\|_{\infty,q}=\,\text{esssup}\|f(\cdot)\|_q~, \textrm{ for } p=\infty,
\]
becomes a Banach space.
\begin{rem}
In the case $p=q=2$ we have the natural isomorphisms 
\[L^2(\bR; L^2(\bR))\cong  L^2(\bR)\otimes L^2(\bR) \cong L^2(\bR^2)\].
\end{rem}

For the rest of the subsection, the exponents $p,q$ lie on the open interval $(1,\infty)$. 
Given $f_1,f_2,\dots, f_n\in L^p(\bR)$ and $g_1,g_2,\dots, g_n\in L^q(\bR)$, define
\[
f: \bR\rightarrow L^q(\bR): f(x)\mapsto \sum\limits_{k=1}^{n}f_k(x)g_k.
\]

We denote this function by $ \sum\limits_{k=1}^{n}f_k\otimes g_k$ and we write $\F(\bR;L^q(\bR))$ for the subspace of $L^p( \bR ; L^q(\bR))$ formed by such functions. Finally, we write $\F(\bR;\S(\bR))$ for the set of functions \\$\sum\limits_{k=1}^{n}f_k\otimes\chi_{A_k}$, where $\{A_k\}_{k=1,\dots,n}$ is a partition of the real line.
\begin{prop} \label{dense1}
The following sets are dense in $L^p(\bR;L^q(\bR))$. 
\begin{enumerate}
\item $\S(\bR;L^q(\bR))\cap L^p(\bR;L^q(\bR))$; 
\item $\F(\bR;L^q(\bR))\cap L^p(\bR;L^q(\bR))$;
\item $\F(\bR;\S(\bR))\cap L^p(\bR;L^q(\bR))$.
\end{enumerate}
\end{prop}
\begin{proof}

The argument for the density of the first two sets can be found in \cite{pis}. For the last set it suffices to prove that given $f\in L^p(\bR), g\in L^q(\bR)$, we can find a sequence $(f_n)$ of elements in  
$\F(\bR;\S(\bR))\cap L^p(\bR;L^q(\bR))$, that converges to $f\otimes g$ with respect to the $\|\cdot\|_{p,q}$ norm. By the classical theory of $L^q$ spaces, there is a sequence of simple functions 
\[g_n=\sum\limits_{k=1}^{n} a_n \chi_{A_n},\, a_n\in \bC,\]
such that $g_n\rightarrow g$ in $L^q(\bR)$. Then the functions $f\otimes g_n$ lie in $\F(\bR;\S(\bR))\cap L^p(\bR;L^q(\bR))$ and 
\[ \|f\otimes g_n - f\otimes g\|_{p,q}=\|f\|_p \|g_n-g\|_q\rightarrow 0. \] 
\end{proof}

The characterization of the dual space of $L^p(\bR;L^q(\bR))$ is again analogous
to the scalar valued case, after we take account of duality in the range space $L^q(\bR)$ (see \cite{pis}).
\begin{prop}\label{dual}
Let $p,q\in(1,\infty)$ be conjugate exponents. The dual space of $L^p(\bR;L^q(\bR))$ is isometrically isomorphic to $L^q(\bR;L^p(\bR))$.
\end{prop}

\begin{lemma}\label{extension}
Given an operator $T\in B(L^q(\bR))$, there is a unique  bounded linear operator
\[\widetilde{T}: L^p(\bR;L^q(\bR))\rightarrow L^p(\bR; L^q(\bR))\]
such that given $f\otimes g \in \F(\bR;L^q(\bR))$
\[ \widetilde{T}(f\otimes g)= f\otimes Tg.\] 
Moreover, the map $T\mapsto \widetilde{T}$ is isometric. 
\end{lemma}
\begin{proof}
Let $f=\sum\limits_{k=1}^{n} \chi_{A_k}\otimes g_k$, such that $g_k\in L^q(\bR)$ and $\{A_k\}_{k=1,\dots,n}$ are pairwise disjoint Borel sets. By linearity, calculate 
\begin{align*}
 \|\widetilde{T}f\|_{p,q}^p &= \bigg\|\sum\limits_{k=1}^{n} \chi_{A_k}\otimes Tg_k\bigg\|_{p,q}^p=
 \int_{\bR} \bigg\|\sum\limits_{k=1}^{n}\chi_{A_k}(x)Tg_k\bigg\|_q^p dx=\\&=
\int_{\bR}\left(\int_{\bR}\bigg|\sum\limits_{k=1}^{n}\chi_{A_k}(x)Tg_k(y)\bigg|^q dy\right)^{p/q}dx=
\int_{\bR}\left(\int_{\bR}\sum\limits_{k=1}^{n}\chi_{A_k}(x)|Tg_k(y)|^q dy\right)^{p/q}dx=\\&=
%\int_{\bR}\left(\sum\limits_{k=1}^{n}\chi_{A_k}(x)\int_{\bR}|Tg_k(y)|^q dy\right)^{p/q}dx=
\int_{\bR}\left(\sum\limits_{k=1}^{n}\chi_{A_k}(x)\|Tg_k\|_q^q\right)^{p/q}dx\leq
\int_{\bR}\left(\sum\limits_{k=1}^{n}\chi_{A_k}(x)\|T\|^q\|g_k\|_q^q\right)^{p/q}dx=\\&=
\|T\|^p \int_{\bR}\left(\sum\limits_{k=1}^{n}\chi_{A_k}(x)\|g_k\|_q^q\right)^{p/q}dx=
\|T\|^p \int_{\bR}\left(\sum\limits_{k=1}^{n}\chi_{A_k}(x)\int_{\bR}|g_k(y)|^q dy\right)^{p/q}dx=\\&=
\|T\|^p \int_{\bR}\left(\int_{\bR}\bigg|\sum\limits_{k=1}^{n}\chi_{A_k}(x)g_k(y)\bigg|^q dy\right)^{p/q}dx=
\|T\|^p\|f\|_{p,q}^p.
\end{align*}
Since the set $\S(\bR;L^q(\bR))$ is dense in $L^p(\bR;L^q(\bR))$, the operator $\widetilde{T}$ is bounded.
To show that the mapping $T\mapsto \widetilde{T}$ is isometric, check that given $g\in L^q(\bR)$
\begin{align*}
\|\chi_{[0,1]}\otimes f\|_{p,q}^p= \int_{\bR} \|\chi_{[0,1]}(x)g\|_q^p dx=\int_0^1 dx\,\|g\|_q^p=\|g\|_q^p
\end{align*}
This yields an upper bound for the norm of the operator $T$
\begin{align*}
\|Tg\|_q =\|\chi_{[0,1]}\otimes Tg\|_{p,q}^p=\|(\widetilde{T}\chi_{[0,1]}\otimes g)\|_{p,q}^p\leq \|\widetilde{T}\|\,\|\chi_{[0,1]}\otimes g\|_{p,q}^p = \|\widetilde{T}\|\, \|g\|_q,
\end{align*}
so the proof is complete.
\end{proof}

\begin{lemma}
Let $p,q\in(1,\infty)$. The linear map 
\[ \Theta : L^p(\bR;L^q(\bR))\rightarrow L^p(\bR;L^q(\bR)): \Theta (f)(x)(y)\mapsto f(x)(x-y) \]
 is a bijective isometry onto $L^p(\bR;L^q(\bR))$.
 \end{lemma}
 \begin{proof}
 It suffices again to consider $f\in \F(\bR;\S(\bR))\cap L^p(\bR;L^q(\bR))$.  Let $f(x)=\sum\limits_{k=1}^{n} f_k(x)\chi_{A_k}$ as before. First, in order to obtain that $\Theta  f$ is strongly measurable, it suffices to show that given $\omega\in L^q(\bR)$, the function
\begin{align*}
\omega(\Theta f(\cdot))=\bR\rightarrow \bC : x\mapsto \omega(\Theta f(x))=\int_{\bR} \omega(y)(\Theta f)(x)(y)dy
\end{align*}
is measurable. This is trivial to prove, since
\begin{align*}
\omega(\Theta f(x))&=\int_{\bR} \omega(y)\sum\limits_{k=1}^{n} f_k(x)\chi_{A_k}(x-y)dy=\\&=
\sum\limits_{k=1}^{n} f_k(x)\int_{\bR} \omega(y)\chi_{A_k}(x-y)dy=\sum\limits_{k=1}^{n} f_k(x)(\omega\ast\chi_{A_k})(x),
\end{align*}
  and applying Young's inequality, the function $\omega\ast\chi_{A_k}$ lies in $L^q(\bR)$. Now
 \begin{align*}
 \|\Theta  f\|_{p,q}^p&= \int_{\bR} \|\Theta (f)(x)\|_q^p dx=
 \int_{\bR} \bigg\|\sum\limits_{k=1}^{n}f_k(x) \chi_{A_k}(x-\cdot)\bigg\|_q^pdx=\\&=\int_{\bR}\sum\limits_{k=1}^{n}|f_k(x)|^p
\|\chi_{A_k}(x-\cdot)\|_q^p dx=
\int_{\bR}\sum\limits_{k=1}^{n}|f_k(x)|^p\|\chi_{A_k}\|_q^p dx=\|f\|_{p,q}^p
  \end{align*}
  Since $\Theta ^{-1}=\Theta $, the map is bijective.
  \end{proof}
\medskip 
\subsection{The Fourier binest algebra $\A_{FB}^p$}
In this subsection, we give the natural generalization of the Fourier binest algebra on  $L^p$ spaces. The Volterra nest  $\mathcal{N}_v^p$ is the continuous nest consisting of the subspaces $L^p([t,+\infty))$, for $t\in\bR$, together with the trivial subspaces $\{0\},L^p(\bR)$. The analytic nest $\N_a^p$ is defined to be the chain of subspaces 
\[
e^{i\lambda x}H^p(\bR), \quad \lambda \in \bR,
\] 
together with the trivial subspaces. These nests determine  the Volterra nest algebra $\A_v^p=Alg\N_v^p$ and the analytic nest algebra $\A_a^p=Alg\N_a^p$, both of which are reflexive operator algebras.

The Fourier binest is the subspace lattice
 \begin{align*}
 \L_{FB}^p=\N_v^p\cup\N_a^p
 \end{align*}
 and the Fourier binest algebra $\A_{FB}^p$ is the non-selfadjoint algebra $Alg \L_{FB}^p$ of operators which leave invariant each subspace of $\L_{FB}^p$. It is elementary to check that $\A_{FB}^p$ is a reflexive algebra, being the intersection of two reflexive algebras.
  
  Given $p\in(1,+\infty)$, let $J$ be the flip operator given by $(Jf)(x)=f(-x)$. Note that $J$ is the isometric operator that takes the Volterra nest to its counterpart \[(\N_v^p)^\perp = \{0\}\cup\{L^p(-\infty,t], t\in\bR\}\cup\{L^p(\bR)\}\] and the analytic nest to
\[(\N_a^p)^\perp= \{0\}\cup\{e^{-i\lambda x}\overline{H^p}(\bR), \lambda\in\bR\}\cup\{L^p(\bR)\}\]. Hence
 $J\A_{FB}^p J$ is the binest algebra generated by the lattice $J\L_{ FB}^p= (\N_v^p)^\perp \cup (\N_a^p)^\perp$.
Since the spaces $e^{i\lambda x}H^p(\bR)$ and $L^p[t,\infty)$ are naturally complemented and have trivial subspaces
it is straightforward to adjust the Hilbert space arguments (\cite{kat-pow-1})to see that $\A_{FB}^p$ is an antisymmetric operator algebra, meaning that $\A_{FB}^p\cap J\A_{FB}^pJ=\bC I$, and contains no non-zero finite rank operators.

\medskip
\subsection{The parabolic algebra} 
We first recall the definition of the strong operator topology (SOT).
Given a net $(T_i)_{i\in I}$ of bounded operators on a Banach space $X$, we say that $T_i\stackrel{SOT}{\rightarrow} T$, where $T\in B(X)$, if and only if $T_i x\rightarrow Tx$, for every $x\in X$. In other words, the SOT topology on $B(X)$ is defined as the topology of pointwise convergence on $X$.

The parabolic algebra $\A_{par}^p$ is defined as the SOT-closed operator algebra 
on $L^p(\bR)$ that is generated by the two isometric semigroups $\{M_\lambda, \lambda\geq 0\}$, $\{D_\mu,\,\mu\geq 0\}$. Since the generators of $\A_{par}^p$ leave the subspaces of the binest $\L_{FB}^p$ invariant, we have 
 $\A_{par}^p\subseteq\A_{FB}^p$. Katavolos and Power showed in \cite{kat-pow-1} that, in the case $p=2$, these two algebras are equal.
\medskip
\subsection{Integral Operators on $L^p(\bR)$}
Let $p\in(1,\infty)$ and $q$ be its conjugate exponent. 
Given $k\in L^p(\bR; L^q(\bR))$, the linear map
\[
(Intk\,f)(x)=\int_\bR k(x)(y)f(y)dy
\]
 defines a bounded operator on $L^p(\bR)$.  Indeed, given $f\in L^p(\bR)$, applying the H\"older inequality we obtain
 \begin{align*}
 \int_{\bR}\bigg| \int_{\bR} k(x)(y) f(y)dy\bigg|^p dx\leq \|k\|_{p,q}^p \|f\|_p^p
\end{align*}

We will refer to such an operator as $(p,q)$-integral operator and
denote the set of $(p,q)$- integral operators  by
\[
\G^p=\{ Intk\,:\, k\in L^p(\bR;L^q(\bR))\}
\]

\begin{rem}
\begin{enumerate}
\item The above calculation also proves that the norm $\|\cdot\|_{p,q}$ dominates the operator norm, i.e. given $(k_n)_{n\geq 1}, k\in 
L^p(\bR; L^q(\bR))$, such that $k_n\stackrel {\|\cdot\|_{p,q}}{\rightarrow} k$, then $Intk_n\rightarrow Intk$.
\item In the special case $p=2$, then $\G^2=\C_2$, where $\C_2$ is the ideal of the Hilbert- Schmidt operators on $L^2(\bR)$.
\end{enumerate}
\end{rem}

\begin{lemma}\label{ideal}
$\G^p$ is a right ideal in $B(L^p(\bR))$.
\end{lemma}
\begin{proof}
Let $T\in B(L^p(\bR))$. Given $f\in L^p(\bR)$ and $k\in L^p(\bR;L^q(\bR))$, such that \\ $k=\sum\limits_{\kappa=1}^{n} f_{\kappa}\otimes g_{\kappa}$, we have
\begin{align*}
 (Intk T f)(x)=\int_{\bR}k(x,y) (Tf)(y)dy&=
\sum\limits_{\kappa=1}^{n} f_{\kappa}(x)\int_{\bR}g_{\kappa}(y) (Tf)(y)dy=\\&=
\sum\limits_{\kappa=1}^{n} f_{\kappa}(x)\int_{\bR}T^\ast g_{\kappa}(y) f(y)dy.
\end{align*}
where $T^\ast$ is the adjoint operator of $T$. Therefore $Intk T=Int \widetilde{k}$, where $\widetilde{k}= \sum\limits_{\kappa=1}^{n} f_{\kappa}\otimes T^\ast g_{\kappa}$.
In the general case, let $k\in L^p(\bR;L^q(\bR))$ and $k_m=\sum\limits_{\kappa=1}^{n} f_{\kappa}^{(m)}\otimes g_{\kappa}^{(m)}$, such that $k_m\stackrel{\|\cdot\|_{p,q}}{\rightarrow} k$. Applying the above argument, we have $Intk_m T= Int\widetilde{k_m}$, where
$\widetilde{k_m}= \sum\limits_{\kappa=1}^{n} f_{\kappa}^{(m)}\otimes T^\ast g_{\kappa}^{(m)}$. Then,  
by  Lemma \ref{extension}, there is a unique operator $\widetilde{T^\ast}\in B(L^p(\bR;L^q(\bR))$, such that 
 \begin{align*}
 \|\widetilde{k_m}-\widetilde{k_l}\|_{p,q}&=\bigg\|\sum\limits_{\kappa=1}^{n} f_{\kappa}^{(m)}\otimes T^\ast g_{\kappa}^{(m)}-\sum\limits_{\kappa=1}^{n} f_{\kappa}^{(l)}\otimes T^\ast g_{\kappa}^{(l)}\bigg\|_{p,q}=\\&=
 \bigg\|\widetilde{T^\ast}\left(\sum\limits_{\kappa=1}^{n} f_{\kappa}^{(m)}\otimes  g_{\kappa}^{(m)}-\sum\limits_{\kappa=1}^{n} f_{\kappa}^{(l)}\otimes  g_{\kappa}^{(l)}\right)\bigg\|_{p,q}\leq\\&
 \leq \big\|\widetilde{T^\ast}\big\|\,\bigg\|\left(\sum\limits_{\kappa=1}^{n} f_{\kappa}^{(m)}\otimes  g_{\kappa}^{(m)}-\sum\limits_{\kappa=1}^{n} f_{\kappa}^{(l)}\otimes  g_{\kappa}^{(l)}\right)\bigg\|_{p,q}=\\&
  =\big\|\widetilde{T^\ast}\big\|\, \|k_m-k_l\|_{p,q}
\end{align*}
It follows that the sequence $(\widetilde{k_m})_m$ is a Cauchy sequence, so by the completeness of $L^p(\bR;L^q(\bR))$, it converges to some $\widetilde{k}\in L^p(\bR;L^q(\bR))$. Since the $\|\cdot\|_{p,q}$ norm dominates the operator norm, the sequence $(Int\widetilde{k_n})_n$ of (p,q)-integral operators converges to $Int\widetilde{k}$. Thus, by the uniqueness of the limit, we obtain $Intk T=Int\widetilde{k}$.
\end{proof}
\medskip

\section{Reflexivity}
In this section, we prove that the parabolic algebra $\A_{par}^p$ is reflexive, given $p\in(1,\infty)$. In particular, we will show that $\A_{par}^p=\A_{FB}^p$. As we noted in the previous section, it suffices to prove that $\A_{FB}^p\subseteq \A_{par}^p$.
\begin{prop}\label{prop1}
Let $Intk\in \G^p\cap \A_{FB}^p$. Then  $k$ satisfies the following properties 
\begin{enumerate}
\item $\Theta (k)\in L^p(\bR;L^q(\bR_+))$
\item For every Borel set $A$ of finite measure and $h\in H^q(\bR)$, we have
\[ \int_{\bR}\left(\int_{\bR} \Theta (k)(x)(y)h(x)\chi_A(y)dy\right)dx=0. \] 
\end{enumerate}
\end{prop}
\begin{proof}
Let $Intk\in \G^p\cap \A_{FB}^p$.
\begin{enumerate}
\item Since $Intk L^p[t,\infty)\subseteq L^p[t,\infty)$, for every $t\in\bR$, it follows that $k(x)(y)=0$, for almost every $(x,y)\in \bR^2$, such that $y>x$. Therefore, $\Theta (k)(x)\in L^q(\bR_+)$ for almost every $x\in\bR$.
\item Since $Intk D_\mu M_\lambda H^p(\bR)\subseteq M_\lambda H^p(\bR)$, for every $\lambda,\mu\in\bR$, given functions 
\\$f\in H^p(\bR), g\in H^q(\bR)$, we have 
\[(Intk D_\mu M_\lambda f)( M_{-\lambda}g)\in H^1(\bR)\]
so
\[\int_{\bR} (Intk D_{\mu}M_\lambda f)(x) (M_{-\lambda}g)(x)dx=0.\]
Therefore, for every $q\in L^1(\bR)$, we obtain 
\[\int_{\bR}\left(\int_{\bR} (Intk D_{\mu}M_\lambda f)(x) (M_{-\lambda}g)(x)dx\right)q(\mu)d\mu=0.\]
Take $q(\mu)=\chi_A(-\mu)$, where $A$ is a Borel set of finite measure.
Then, by Fubini's theorem 
\begin{align*}
&\int_{\bR}\left(\int_{\bR}\left(\int_{\bR} k(x)(y) e^{i\lambda(y-\mu)}f(y-\mu)e^{-i\lambda x} g(x)dy\right)dx\right)\chi_A(-\mu)d\mu=0
\Leftrightarrow\\
&\int_{\bR}\left(\int_{\bR}\left(\int_{\bR}  e^{i\lambda(y-\mu)}f(y-\mu)\chi_A(-\mu)d\mu\right)k(x)(y)e^{-i\lambda x} g(x) dy\right)dx=0
\stackrel{\mu\rightarrow \mu+y-x}{\Leftrightarrow}\\&
\int_{\bR}\left(\int_{\bR}\left(\int_{\bR}  e^{i\lambda(x-\mu)}f(x-\mu)\chi_A(x-y-\mu)d\mu\right)k(x)(y)e^{-i\lambda x} g(x) dy\right)dx=0
\Leftrightarrow\\&
\int_{\bR}\left(\int_{\bR}\left(\int_{\bR} k(x)(y) D_\mu f(x)g(x)D_\mu\chi_A(x-y)dy\right) dx\right)e^{-i\lambda\mu}d\mu=0 
\end{align*}

We claim that the function
\[
\Phi:\bR\rightarrow\bC: \mu\mapsto \int_{\bR}\left(\int_{\bR} k(x)(y) D_\mu f(x)g(x)D_\mu\chi_A(x-y)dy\right) dx
\]
is a well defined $L^1$ function. By Tonelli's theorem, it suffices to show that 
\[ \int_{\bR}\left(\int_{\bR}\left(\int_{\bR} \big|k(x)(y) D_\mu f(x)g(x)D_\mu\chi_A(x-y)\big|d\mu\right) dy\right) dx <\infty\].
We have
\begin{align*}
&\int_{\bR}\left(\int_{\bR}\left(\int_{\bR} \big|k(x)(y) D_\mu f(x)g(x)D_\mu\chi_A(x-y)\big|d\mu\right) dy\right) dx =\\&=
\int_{\bR}\left(\int_{\bR}\left(\big|k(x)(y)\big|\int_{\bR} \big| f(x-\mu)\chi_A(x-y-\mu)\big|d\mu\right) dy\right) \big|g(x)\big|dx=\\&=
\int_{\bR}\left(\int_{\bR}\left(\big|k(x)(y)\big|\int_{\bR} \big| f(\mu)\chi_A(\mu-y)\big|d\mu\right) dy\right) \big|g(x)\big|dx=\\&=
\int_{\bR}\left(\int_{\bR}\left(\big|k(x)(y)\big|\int_{\bR} \big| f(\mu)J\chi_A(y-\mu)\big|d\mu\right) dy\right) \big|g(x)\big|dx,
\end{align*}
where $J$ is again the flip operator. By Young's inequality the function $h:=|f|\ast J\chi_A$ lies in $L^p(\bR)$, so the expression above is equal to 
\[
\int_{\bR}\left(\int_{\bR}\big|k(x)(y)h(y)\big| dy\right) \big|g(x)\big|dx 
\]
which by Holder's inequality is bounded by $\|h\|_p\|k\|_{p,q}\|g\|_q$, so our claim is proven. 
Since the Fourier transform of the function $\Phi$ is the zero function, it follows that for almost every $\mu\in \bR$
\[
\int_{\bR}\left(\int_{\bR} k(x)(y) D_\mu f(x)g(x)D_\mu\chi_A(x-y)dy\right) dx=0
\]
Hence by Lemma \ref{dense}
\[
\int_{\bR}\left(\int_{\bR} k(x)(y) h(x)D_\mu\chi_A(x-y)dy\right) dx=0
\]
for every $h\in H^q(\bR)$. Moreover, since the Borel set $A$ was freely chosen and $D_\mu\chi_A=\chi_{A+\mu}$,
where $A+\mu=\{x+\mu:x\in A\}$, we have the equivalence  
\[
\int_{\bR}\left(\int_{\bR} k(x)(y) h(x)\chi_A(x-y)dy\right) dx=0\\
\stackrel{y\rightarrow x-y} {\Leftrightarrow} \int_{\bR}\left(\int_{\bR} \Theta (k)(x)(y) h(x)\chi_A(y)dy\right) dx=0
\]
\end{enumerate}
\end{proof}
Our next goal is to determine a dense set of $\G^p\cap\A_{FB}^p$. We start with an approximation lemma.
\begin{lemma} 
Let $\phi\in L^1(\bR)$. Then, given $p\in[1,\infty)$, the convolution operator 
\[\Delta_\phi: L^p(\bR)\rightarrow L^p(\bR):f\mapsto \phi\ast f,\]
is bounded. Furthermore, if $\phi$ has essential support in $\bR_+$, then $\Delta_\phi$ belongs to the SOT-closed algebra generated by $\{D_t\,|\, t\in \bR_+\}$.
\end{lemma}
\begin{proof}
The continuity of $\Delta_\phi$ is immediate by Young's inequality, which also gives $\|\Delta_\phi\|\leq \|\phi\|_1$. The argument of the second claim is similar to that for $p=2\,$ \cite{lev}.  Suppose first that $\phi$ has compact support $[a,b]$, for some $b>\alpha\geq 0$. Given $n\in\bN$ and $m\in{0,1,\dots,n-1}$,  define $\alpha_{m,n}=\int_{\tau(m,n)}^{\tau(m+1,n)}\phi(s)ds$, where $\tau(m,n)=a+\frac{m}{n}(b-a)$. We claim that the sequence $(T_n)_n$ given by 
\[
T_n=\sum_{m=0}^{n-1} \alpha_{m,n} D_{\tau(m,n)}
\]
converges in the SOT-topology to $\Delta_\phi$. Consider $f\in L^p$. Then by Hahn - Banach theorem
\begin{align*}
\|(D_\phi-T_n)f\|_p&= sup\left\{\bigg|\int_{\bR}(D_\phi-T_n)f(x)g(x)dx\bigg| \,:\, \|g\|_q=1\right\}=\\&=
sup\left\{\bigg|\int_{\bR}\sum_{m=0}^{n-1}\int_{\tau(m,n)}^{\tau(m+1,n)}\phi(t) \left( (D_t-D_{\tau(m,n)})f(x)\right)dt g(x)dx\bigg| \,:\, \|g\|_q=1\right\}=\\&=
sup\left\{\bigg|\int_{\bR}\int_{\bR}\phi(t) \left( (D_t-D_{\rho_n(t)})f(x)\right)dt g(x)dx\bigg| \,:\, \|g\|_q=1\right\}\leq\\&\leq
sup\left\{\int_{\bR}\left(\int_{\bR}\big|\phi(t) \left( (D_t-D_{\rho_n(t)})f(x)\right) g(x)\big| dx\right)dt \,:\, \|g\|_q=1\right\}
\end{align*}
where $\rho_{n}(t)=a+\frac{b-a}{n}\left\lfloor \frac{(t-a)n}{b-a}\right\rfloor\,,\, t\in[a,b]$. Now 
\[
\int_{\bR} \big|\phi(t) \left( (D_t-D_{\rho_n(t)})f(x)\right) g(x)\big| dx \leq \|(D_t-D_{\rho_n(t)})f\|_p \|g\|_q, \]
so it follows that
\[
\|(D_\phi-T_n)f\|_p\leq \int_{\bR} |\phi(t)|\|(D_t-D_{\rho_n(t)})f\|_p dt.
\]
Since $\|(D_t-D_{\rho_n(t)})f\|_p\rightarrow 0$ as $n\rightarrow \infty$ and $|\phi(t)|\|(D_t-D_{\rho_n(t)})f\|_p\leq 2|\phi(t)\|f\|_p$, we get that $\|(D_\phi-T_n)f\|_p\rightarrow 0$, by dominated convergence theorem.

The general case, is a simple application of Young's inequality.
\end{proof}
\begin{rem}
In the $L^2(\bR)$ case, there is a simpler proof, using the unitary Fourier-Plancherel transform $F$. Note that
\[\Delta_\phi=F^\ast M_{\hat{\phi}}F\]
 Since $\phi\in L^1(\bR_+)$, it follows that $\hat{\phi}\in\overline{H^\infty(\bR)}$. Therefore, the multiplication operator $M_{\hat{\phi}}$ lies in the SOT-closed algebra generated by $\{M_{-\lambda}\,:\, \lambda\in\bR_+\}$. Hence, using the fact that  $D_{\lambda}=F^\ast M_{-\lambda}F$, the proof is complete. 
\end{rem}
%\begin{align*}
%(F\Delta_\phi F^\ast f)(x)
%&=\int_\bR e^{-ix\xi}\Delta_\phi F^\ast f(\xi)d\xi=\\&=
%\int_\bR e^{-ix\xi}\left(\int_\bR \phi(\xi-y) F^\ast f(y)dy\right)d\xi=\\&=
%\int_\bR e^{-ix\xi}\left(\int_\bR \phi(\xi-y)\left(\int_R e^{iy\omega} f(\omega) d\omega\right)dy\right)d\xi=\\&=
%\int_\bR e^{-ix\xi}\left(\int_\bR\left(\int_R \phi(\xi-y)e^{iy\omega}dy\right)  f(\omega) d\omega\right)d\xi=\\&=
%\int_\bR e^{-ix\xi}\left(\int_\bR e^{i\xi\omega}\left(\int_R \phi(y)e^{-iy\omega}dy\right)  f(\omega) d\omega\right)d\xi=\\&=
%\int_\bR e^{-ix\xi}\left(\int_\bR e^{i\xi\omega}(\hat{\phi}f)(\omega) d\omega\right)d\xi=\\&=
%(FF^\ast(\hat{\phi}f)(x)
%=(\hat{\phi}f)(x).
%\end{align*}
%Hence on the dense subspace $(L^2\cap L^p)(\bR)$ of $L^p(\bR)$, we obtain 
%,\]
%where $J$ is the flip operator, given by $(Jf)(x)=f(-x)$.
%Since $\phi\in L^1(\bR_+)$, it follows that $\hat{\phi}\in\overline{H^\infty(\bR)}$, which yields %that $J\hat{\phi}\in H^\infty(\bR)$.
%Therefore, $\Delta_\phi\in SOT-alg\{D_t\,|\, t\in \bR_+\}$.

\begin{lemma}
Let $h\in H^p(\bR),\, \phi\in L^q(\bR_+)$, where $p\in(1,\infty)$ and $q$ is its conjugate exponent. Define $k=\Theta ^{-1}(h\otimes \phi)$. Then, the operator $Intk=M_h\Delta_\phi$
lies in $\G^p\cap\A_{par}^p$.
\end{lemma}
\begin{proof}
First, consider $h\in H^\infty(\bR), \phi\in L^1(\bR_+)$. Then
\begin{align*}
(Intk f)(x)&=\int_\bR \Theta ^{-1}(h\otimes\phi)(x)(y)f(y)dy=\\&=\int_{\bR}h(x)\phi(x-y)f(y)dy=(M_h\Delta_\phi f)(x),
\end{align*}
so $\|Intk\|\leq \|h\|_\infty \|\phi\|_1$. By the previous lemma $\Delta_\phi\in SOT-alg\{D_t\,:\,t\in \bR_+\}$, hence $Intk\in\A_{par}^p$. 
Take now $h\in H^p(\bR)$ and $\phi\in L^q(\bR_+)$. Then there exist $h_m\in (H^\infty\cap H^p)(\bR)$, 
$\phi_m\in (L^1\cap L^q)(\bR_+)$, such that $h_m\stackrel{\|\cdot\|_p}{\rightarrow}h$ and $\phi_m\stackrel{\|\cdot\|_q}{\rightarrow} \phi$. Now it is straightforward to show that $h_m\otimes\phi_m\stackrel{\|\cdot\|_{p,q}}{\rightarrow}h\otimes\phi$. Since the norm $\|\cdot\|_{p,q}$ dominates the operator norm and $\A_{par}^p$ is norm closed, 
\[Intk=Int(\Theta ^{-1}(h\otimes \phi))\in \A_{par}^p.\]
Moreover, the fact that $h$ and $\phi$ lie in $H^p(\bR)$ and $L^q(\bR_+)$ respectively implies that $Intk\in \G^p$.
\end{proof}

\begin{prop}\label{babyequality}
$\G^p\cap\A_{FB}^p=\G^p\cap\A_{par}^p$, for every $p\in(1,\infty)$.
\end{prop}
\begin{proof}
By Proposition \ref{prop1}, if $Intk\in \G^p\cap \A_{FB}^p$, then $\Theta (k)\in L^p(\bR;L^q(\bR_+))$
and 
\begin{equation}\label{eq1}
 \int_{\bR}\left(\int_{\bR} \Theta (k)(x)(y)\eta(x)\chi_A(y)dy\right)dx=0,
 \end{equation}
 for every Borel set $A$ of finite measure and $\eta\in H^q(\bR)$. 
 
 By the Hahn - Banach theorem and Proposition \ref{dual}, it suffices to prove, that given $\omega\in (L^p(\bR;L^q(\bR)))^\ast\cong L^q (\bR; L^p(\bR))$, such that $\omega(\Theta ^{-1}(h\otimes\phi))=0$, for every $h\in H^p(\bR), \phi\in L^q(\bR_+)$, then $\omega(k)=0$, for every kernel $k$ that corresponds to an operator $Intk\in\G^p\cap \A_{FB}^p$. 
 It suffices to check for $\omega$ in a dense subset of $L^q(\bR;L^p(\bR_+))$, so by Proposition \ref{dense1}, for all $\omega\in  
 \Theta (\F(\bR;\S(\bR))\cap L^p(\bR;L^q(\bR)))$. We have
 \begin{align*}
 \omega(k)&=\sum\limits_{m=1}^{n}\int_\bR\left(\int_\bR k(x)(y)f_m(x)\chi_{A_m}(x-y)dy\right)dx=\\&=
 \sum\limits_{m=1}^{n} \int_\bR\left(\int_\bR \Theta k (x)(y) f_m(x)\chi_{A_m}(y)dy\right)dx
 \end{align*} 
 where $f_m\in L^q(\bR)$ and $\{A_m\}_{m=1,\dots,n}$ a family of pairwise disjoint Borel sets. Now
 \begin{align*}
 0=\omega(\Theta ^{-1}(h\otimes\phi))&=\sum\limits_{m=1}^{n}\int_\bR\left(\int_\bR h(x)\phi(x-y)f_m(x)\chi_{A_m}(x-y)dy\right)dx=\\
 &=\sum\limits_{m=1}^{n}\int_\bR\left(\int_\bR \phi(x-y)\chi_{A_m}(x-y)dy\right)h(x)f_m(x)dx=\\
 &=\sum\limits_{m=1}^{n}\int_\bR\left(\int_\bR \phi(y)\chi_{A_m}(y)dy\right)h(x)f_m(x)dx.
\end{align*}
If  $A_m$ lies in $\bR_-$, for every $m=1,\dots,n$, and $\Theta (k)\in L^p(\bR;L^q(\bR_+))$, then $\omega(k)=0$. 
Therefore, we may assume that $A_m\subseteq \bR_+$, for every $m\in\{1,\dots,n\}$. Fix now some $m_0\in\{1,\dots,n\}$ and take 
$\phi=\chi_{A_{m_0}}$. Then
\[\int_\bR h(x)f_{m_0}(x)dx=0,\,\forall h\in H^p(\bR).\]
Therefore, $f_m\in H^q(\bR)$, for every $m$.\[\omega(k)=\sum\limits_{m=1}^{n} \int_\bR\left(\int_\bR \Theta k (x)(y) f_m(x)\chi_{A_m}(y)dy\right)dx=0\]
by equation \eqref{eq1}.
\end{proof}
The following proposition and proof follow the pattern for the case $p=2$, given in \cite{lev}.
\begin{prop}
For every $p\in(1,\infty)$, the algebra $\A_{par}^p$ contains a bounded approximate identity of elements in $\G^p$.
\end{prop}
\begin{proof}
Take $h_n(x)=\frac{ni}{x+ni}$ and $\phi_n(y)=n\chi_{[0,1/n]}(y)$. It is trivial to see that 
$h_n$ and $\phi_n$ lie in $H^r(\bR)$ and $L^r(\bR_+)$, respectively, for every $r\in(1,\infty)$.
Let $k_n=\Theta ^{-1}(h_n\otimes\phi_n)$. Then $\|Intk_n\|\leq \| h_n\|_\infty \|\phi_n\|_1\leq 1$.
Since $h_n\rightarrow 1$ uniformly on compact sets of the real line, it follows that $M_{h_n}\stackrel{SOT}{\rightarrow}I$.
Now given $f\in C_\mathrm{C}(\bR)$, note that
\begin{align*}
\|\Delta_{\phi_n}f-f\|_p^p&=\int_\bR\bigg|\int_\bR n\chi_{[0,1/n]}(y)f(x-y)dy-f(x)\bigg|^p dx=\\&=
\int_\bR\bigg|\int_0^{1/n}n f(x-y)dy-f(x)\bigg|^p dx
\end{align*}
Check that $\bigg|\int_0^{1/n}n f(x-y)dy-f(x)\bigg|^p\leq 2^p\|f\|_\infty^p \chi_s(x)$, where $S$ is the compact set 
\[S=\{x+\tau\,|\, x\in\textrm{\,supp}f,\,\tau\in[0,1]\}.\]
Hence by dominated convergence $\Delta_{\phi_n}f\rightarrow f$.
Since $C_\mathrm{C}(\bR)$ is dense in $L^p(\bR)$ it follows that $\Delta_{\phi_n}\stackrel{SOT}{\rightarrow} I$.
Multiplication is $SOT$-continuous on the closed unit ball of bounded operators, so $M_{h_n}\Delta_{\phi_n} \stackrel{SOT}{\rightarrow} I$.
\end{proof}

\begin{thm} 
For every $p\in(1,\infty)$, the parabolic algebra $A_{par}^p$ is equal to the Fourier binest algebra $\A_{FB}^p$.
\end{thm}
\begin{proof}
As we have noted before, it suffices to prove that $\A_{FB}^p\subseteq A_{par}^p$. Let $T\in \A_{FB}^p$ and $(X_n)_{n\geq1}$ be the bounded approximate identity of the previous proposition. 
By lemma \ref{ideal} and proposition \ref{babyequality}, the operators $X_nT$ lie in $\G^p\cap \A_{FB}^p=\G^p\cap \A_{par}^p$.
Since $\A_{par}^p$ is $SOT$-closed, the given operator $T=SOT-\lim\limits_n X_nT$ lies in $\A_{par}^p$.
\end{proof}

\begin{prop}
The Fourier binest algebra $\A_{FB}^\infty$ is strictly larger than the parabolic algebra $\A_{par}^\infty$.
\end{prop}
\begin{proof}
Note first that the non-selfadjoint $\|\cdot\|_\infty$-closed algebra of the trigonometric polynomials $\{e^{i\lambda x}:\lambda\geq 0\}$ is the algebra $AAP(\bR)$ of analytic almost periodic functions \cite{bes}, which is strictly smaller than $H^\infty(\bR)$. Take a function $\phi$ that lies in $H^\infty(\bR)$ and it is not an element of $AAP(\bR)$. It suffices to show that $M_\phi\notin \A_{par}^\infty$. If this is not the case, there is some sequence $p_n(M_\lambda, D_\mu)$ in the non-closed algebra generated by $\{M_\lambda, D_\mu\,:\, \lambda,\mu\geq0\}$ which converges strongly to $M_\phi$. Thus for any $f\in L^\infty(\bR)$, we have
\[
\bigg\|p_n(M_{\lambda}, D_{\mu})f-M_\phi f\bigg\|_\infty\rightarrow 0, \text{ as} n\rightarrow\infty. 
\]
Choosing $f\equiv 1$, it follows that 
\[
\bigg\|p_n(M_{\lambda}, I)f-M_\phi f\bigg\|_\infty\rightarrow 0, \text{ as} n\rightarrow\infty. 
\]
and so $\phi\in AAP(\bR)$, a contradiction.
\end{proof}

\begin{rem}
It remains unclear to the author whether the parabolic operator algebras $\A_{par}^{1}$ and $\A_{par}^{\infty}$ acting on the respective Banach spaces $L^1(\bR)$ and $L^\infty(\bR)$  are reflexive operator algebras.
\end{rem}

\medskip
\section{The lattice of the parabolic algebra}
 
Let
$K_{\lambda,s}^p=M_\lambda M_{\phi_s}H^p(\bR)$ where $\phi_s(x)=e^{-isx^2/2}$. This is evidently an invariant subspace for the multiplication semigroup and for $s\geq 0$ one can check that it is invariant for the translation semigroup. Thus
for $s\geq 0$ the nest 
$\N_s=M_{\phi_s}\N_a$ is contained in $Lat \A_{par}^p$ and these nests are distinct. In fact any two nontrivial subspaces from nests with distinct $s$ parameter have trivial intersection. 

Suppose now that $p=2$. With the strong operator topology for the associated orthogonal subspace projections it can be shown (\cite{kat-pow-1}) that the set of these nests for $s\geq 0$, together with the Volterra nest $\N_v^2$, is homeomorphic to the closed unit disc.
% with the topological boundary being the binest $\L_{FB}$. 
A cocycle argument given in   \cite{kat-pow-1} leads to the fact that every invariant subspace for $\A_{par}^2$ is of this form for $p=2$. That is

\begin{equation}\label{LatAp}
Lat\A_{par}^2=\{K_{\lambda,s}^2|\lambda\in\bR,s\geq 0\}\cup \N_v
\end{equation}

We prove now the corresponding result for the general case of $\A_{par}^p$, where $1< p<\infty$. 
\medskip
%\begin{figure}[h!]
%\begin{center}
%\begin{tikzpicture}
%\filldraw [black] (3,6) circle (1pt);
%\draw (3,6) node [above]  {$L^p(\bR)$};
%\filldraw [black] (3,0) circle (1pt);
%\draw (3,0) node [below] {$\{0\}$};
%\filldraw [black] (0,3) circle (1pt);
%\draw (0,3) node [left]  {$H^p(\bR)$};
%\filldraw [black] (6,3) circle (1pt);
%\draw (6,3) node [right]  {$L^p(\bR_+)$};
%\filldraw [black] (1.65,3) circle (1pt);
%\draw (1.65,3) node [right] {$M_{\phi_s}H^p(\bR)$};
%\filldraw [black] (0.35,1.59) circle (1pt);
%\draw (4,3) node [right] {$s\rightarrow +\infty$};
%\draw (0.35,1.59) node [left]  {$M_{\lambda}H^p(\bR)$};
%\filldraw [black] (5.65,1.59) circle (1pt);
%\draw (5.65,1.59) node [right]  {$D_\mu L^p(\bR_+)$};
%\filldraw [black] (1.9,1.59) circle (1pt);
%\draw (1.9,1.59) node [right]  {$M_{\phi_s}M_\lambda H^p(\bR)$};
%\draw (0.35,4.41) node [above left] {$\N_a$};
%\draw (1.9,4.41) node [above left] {$\N_s$};
%\draw (5.65,4.41) node [above right] {$\N_v$};
%\draw (3,3) circle [radius=3.0];
%\draw (3,0) to [out=140,in=220] (3,6);
%\end{tikzpicture}
%\end{center}
%\caption{Parametrising $Lat \A_{par}^p$ by the unit disc.}
%\label{Latpar}
%\end{figure}
 
 Let $K$ be a non-trivial element of $Lat \A_{par}^p$. By the Beurling theorem, either $K= L^p(E)$ for some Borel set $E\subseteq \bR$ or $K=M_\phi H^p(\bR)$ for some unimodular function $\phi$. On the other hand,  the subspace $K\cap L^2(\bR)$ is invariant under the generators of the parabolic algebra. Therefore, the $\|\cdot\|_2$-closure of $K\cap L^2(\bR)$ lies in $Lat \A_{par}^2$. Hence, according to the Beurling -Lax theorem, we have two cases.
 In the first case, where $K=L^p(E)$, then 
 \[\overline{L^p(E)\cap L^2(\bR)}^{\|\cdot\|_2}\in Lat \A_{par}^2\Rightarrow
 L^2(E)\in Lat \A_{par}^2\Rightarrow E=[t,\infty),\]
 for some $t\in\bR$. In the second case, $K= M_\phi H^p(\bR)$, which implies 
 \[\overline{M_\phi H^p(\bR)\cap L^2(\bR)}^{\|\cdot\|_2}\in Lat \A_{par}^2\Rightarrow
 M_\phi H^2(\bR)\in Lat \A_{par}^2\Rightarrow M_\phi= M_{\phi_s}M_\lambda,\]
 for some $s\in[0,+\infty),\lambda\in \bR$. Hence, we have the following result.
\begin{thm} Given $p\in(1,\infty)$, the invariant subspace lattice of the algebra $\A_{par}^p$ is 
\begin{equation*}\label{LatA_{par}^p}
Lat\A_{par}^p=\{K_{\lambda,s}^p|\lambda\in\bR,s\geq 0\}\cup \N_v^p
\end{equation*}
\end{thm}

Recall that the reflexive closure of a set of closed subspaces $\L$ is the subspace lattice $Lat Alg \L$. Thus the theorem identifies the reflexive closure of the binest $\L_{FB}^p$.

\begin{rem}
In \cite{kat-pow-1}, Katavolos and Power proved that $Lat\A_{par}^2$, viewed as a topological space of projections on $L^2(\bR)$, endowed with the strong operator topology, is homeomorphic to the closed unit disc. In particular, they obtained the so-called strange limit  
\[
P_{K_{s,\lambda}^2}\stackrel{SOT}{\rightarrow} P_{L^2[\lambda,+\infty)}, \text{ as } s\rightarrow \infty,
\]
which relies on the Paley - Wiener theorem and the fact that the Fourier transform is unitary on $L^2(\bR)$. Even though the Riesz projection from $L^p(\bR)$ onto $H^p(\bR)$ remains bounded, it is unknown to the author if the above convergence still holds, for $p\in (1, +\infty)\backslash\{2\}$. 
\medskip

We expect that the operator algebras $\A_{par}^p$, for $1<p<\infty$, are pairwise non isomorphic, even as rings of linear operators. However, the standard methods for such a demonstration (which go back to Eidelheit \cite{ei}) rely on exploiting the presence of rank one operators to deduce an isomorphism between the underlying Banach spaces. Possibly the $(p,q)$-integral operators could once again play such a substitute role.
\end{rem}

\end{document}